\documentclass[12pt,psamsfonts]{amsart}
\usepackage{amsmath,amsthm,amsfonts,amssymb,pstricks}

\usepackage{eucal}
\newcommand{\1}{{1 \hspace{-0.35em} {\rm 1}}}

\newcommand{\tr}{{\rm tr}}

 \DeclareMathOperator{\Sp}{Sp}
\DeclareMathOperator{\End}{End} 
 \DeclareMathOperator{\im}{im}
 
 \DeclareMathOperator{\Id}{Id}

\newtheorem{thm}{Theorem}[subsection]
 
 \newtheorem{lem}[thm]{Lemma}
 
 \theoremstyle{definition}
 \newtheorem{defn}[thm]{Definition}
 \theoremstyle{remark}
 \newtheorem{rmk}[thm]{Remark}

\begin{document}
\title[$Invariants of links from the generalized Yang-Baxter equation$]
{Invariants of links from the generalized Yang-Baxter equation}

\author{Seung-moon Hong}
\email{seungmoon.hong@utoledo.edu}
\address{Department of Mathematics and Statistics\\
    The University of Toledo \\
    Toledo, OH 43606\\
    U.S.A.}


\begin{abstract}Enhanced Yang-Baxter operators give rise to invariants of oriented links. We expand the enhancing method to generalized Yang-Baxter operators. At present two examples of generalized Yang-Baxter operators are known and recently three types of variations for one of these were discovered. We present the definition of enhanced generalized YB-operators and show that all known examples of generalized YB-operators can be enhanced to give corresponding invariants of oriented links. Most of these invariants are specializations of the polynomial invariant $P$. Invariants from generalized YB-operators are multiplicative after a normalization.

\end{abstract}

\maketitle
\section{Introduction}

Solutions to the Yang-Baxter equation are called Yang-Baxter operators and give rise to representations of braid groups in a canonical way. It is well known that every oriented link can be obtained by closing some braid. Based on this relation between braids and links, V. G. Turaev introduced the enhanced Yang-Baxter operators (briefly, EYB-operators) and defined an isotopy invariant $T_S$ of oriented links from each EYB-operator $S$ in \cite{Tu}. In \cite{RZWG} generalized Yang-Baxter operator (briefly, gYB-operator) was proposed and a whole family of $(2,k,2^t)$-type gYB-operators were defined. We discuss a (2,3,2)-type gYB-operator as one of the main examples. Another example of gYB-operator, (2,3,1)-type, appeared in \cite{GHR} and three families of its variations were discussed in \cite{Ch}. In this paper we generalize the Turaev's enhancing method to obtain isotopy invariants of oriented links from the examples of gYB-operators. 

Note that from any ribbon category we have a link invariant for each object. If there is a generalized localization in the sense of \cite{GHR}, then some enhancement should exist and it is reasonable to expect that the corresponding invariant recovers the one defined directly from the category.

Here are the contents of this paper in more detail. In section \ref{gYB} we recall gYB-operator and consider the enhanced generalized Yang-Baxter operators (briefly, EgYB-operators). In section \ref{inv} we define an invariant of oriented links associated with each EgYB-operator in a similar way as in \cite{Tu}. In section \ref{ex} some examples of EgYB-operators and corresponding link invariants are studied.

\vspace{3mm}
\emph{Notation and convention.} In this paper $V$ denotes a finite dimensional vector space over complex number field $\mathbb{C}$. However all discussion is still valid for any finitely generated free module V over a commutative ring with 1. By a link, we mean an oriented link unless otherwise stated as we mainly consider invariants of oriented links in this paper. We denote the identity map of $V$ by $\Id_V$ and identity map of $V^{\otimes m}$ by simply $I_m$ when the vector space $V$ is clear from the context. Each basis $\{v_1,\ldots, v_d\}$ of a d-dimensional vector space $V$ gives us a basis $\{v_{i_1}\otimes \ldots \otimes v_{i_{n}}| i_1, \ldots, i_n \in \{1, 2, \ldots, d\}\}$ of $V^{\otimes n}$. On this basis, each endomorphism $f \in \End(V^{\otimes n})$ can be represented as a mulitiindexed matrix $\left( f^{j_1, \ldots, j_n}_{i_1, \ldots, i_n} \right)_{1 \leq i_1,j_1, \ldots, i_n,j_n \leq d}$ by

$$ f(v_{i_1}\otimes \ldots \otimes v_{i_n})=\sum_{1 \leq j_1, \ldots, j_n \leq d} f^{j_1, \ldots, j_n}_{i_1, \ldots, i_n} v_{j_1}\otimes \ldots \otimes v_{j_n}.$$

\vspace{3mm}\emph{Acknowledgements.} I am grateful to Eric Rowell for his encouragement and useful discussions.

\section{The generalized Yang-Baxter operators}\label{gYB}

\subsection{The enhanced Yang-Baxter operators}\label{EYB} In this subsection, we recall the EYB-operator introduced in \cite{Tu}.

\begin{defn}
An isomorphism $R:V^{\otimes 2} \rightarrow V^{\otimes 2}$ is called a \textbf{Yang-Baxter operator} (briefly, a YB-operator) if it satisfies the following Yang-Baxter equation:
$$(R\otimes I)\circ(I\otimes R)\circ(R\otimes I)=(I\otimes R)\circ(R\otimes I)\circ(I\otimes R):V^{\otimes 3}\rightarrow V^{\otimes 3}$$
\noindent where $I=\Id_V$.
\end{defn}

For each $f \in \End(V^{\otimes n})$ one can define its operator trace $\Sp_n(f)\in \End(V^{\otimes n-1})$. If $\{v_1,\ldots, v_d\}$ is a basis of $V$ then for any $i_1, \ldots, i_{n-1}\in \{1, 2, \ldots, d\}$

$$\Sp_n(f)(v_{i_1}\otimes \ldots \otimes v_{i_{n-1}})=\sum_{1 \leq j_1,\ldots,j_{n-1},j \leq d}f^{j_1,\ldots,j_{n-1},j}_{i_1,\ldots,i_{n-1},j} v_{j_1}\otimes \ldots \otimes v_{j_{n-1}}.$$

Note that $\Sp_n(f)$ does not depend on the choice of basis of $V$ and $\tr(\Sp_n(f))=\tr(f)$ where $\tr$ is the ordinary trace of endomorphisms.

\begin{defn}\label{def-EYB}
An \textbf{enhanced Yang-Baxter operator} (EYB-operator) is a collection \{a YB-operator $R:V^{\otimes 2} \rightarrow V^{\otimes 2}$, $\mu :V \rightarrow V $, invertible elements $\alpha, \beta$ of $\mathbb{C}$ \} which satisfies the following conditions:

(i) The endomorphism $\mu \otimes \mu : V^{\otimes 2} \rightarrow V^{\otimes 2}$ commutes with $R$;

(ii) $\Sp_2(R \circ (\mu \otimes \mu))=\alpha\beta\mu$; $\Sp_2(R^{-1} \circ (\mu \otimes \mu))=\alpha^{-1}\beta\mu.$

\end{defn}

\subsection{Extended operator trace and it's properties}

The above operator trace map $\Sp_n$ is obtained by fixing the last single tensor factor. We may define similar operator trace maps $\Sp_{k,m}:\End(V^{\otimes k})\rightarrow \End(V^{\otimes k-m}), m < k,$ by fixing the last $m$ tensor factors as follows:

 $$\Sp_{k,m}(f)(v_{i_1}\otimes \ldots \otimes v_{i_{k-m}})=\sum_{1 \leq j_1,\ldots,j_{k-m},l_1,\ldots,l_m \leq d}f^{j_1,\ldots,j_{k-m},l_{1}, \ldots, l_{m}}_{i_1,\ldots,i_{k-m},l_{1}, \ldots,l_{m}} v_{j_1}\otimes \ldots \otimes v_{j_{n-1}}.$$

This operator trace map $\Sp_{k,m}$ does not depend on the choice of basis of $V$ and $\tr(\Sp_{k,m}(f))=\tr(f)$ as well because it is simply a composition $\Sp_{k-m+1}\circ \ldots \circ \Sp_{k}$. In this notation, $\Sp_n$ is equal to $\Sp_{n,1}$.  We will use $\Sp_{3,2}$ in the subsection \ref{ex-2}. The following lemma is obtained directly from the definition.

\begin{lem}\label{sp}
For $m<k$, $f \in \End(V^{\otimes k})$ and $g \in \End(V^{\otimes k-m})$,

$$ \Sp_{k,m}(f\circ (g \otimes \Id_{V^{\otimes m}} ))=\Sp_{k,m}(f)\circ g;$$
$$ \Sp_{k,m}((g \otimes \Id_{V^{\otimes m}} )\circ f)=g\circ \Sp_{k,m}(f);$$
$$\Sp_{l+k,m}(\Id_{V^{\otimes l}} \otimes f)=\Id_{V^{\otimes l}} \otimes \Sp_{k,m}(f).$$

\end{lem}

Note that this lemma implies $\Sp_{l+k,m}(h\otimes f)=h\otimes \Sp_{k,m}(f)$ for $h \in \End(V^{\otimes l})$ and $f \in \End(V^{\otimes k})$.

\subsection{The enhanced generalized Yang-Baxter operators}\label{EgYB}

\begin{defn}
An isomorphism $R:V^{\otimes k} \rightarrow V^{\otimes k}$ is called a \textbf{generalized Yang-Baxter operator} (briefly, a gYB-operator) of type $(d,k,m)$ if it satisfies the following generalized Yang-Baxter equation and far-commutativity:

$$(R\otimes I_m)\circ(I_m \otimes R)\circ(R \otimes I_m)=(I_m \otimes R)\circ(R \otimes I_m)\circ(I_m \otimes R);$$
$$(R\otimes I^{\otimes (j-2)}_m)\circ(I^{\otimes (j-2)}_m\otimes R)=(I^{\otimes (j-2)}_m\otimes R)\circ(R\otimes I^{\otimes (j-2)}_m)\:\: \text{for}\:\: j \geq 4$$

\noindent where $\dim(V)=d$ and $I_m = \Id_{V^{\otimes m}}$.
\end{defn}

Note that a $(d,2,1)$ type gYB-operator is the ordinary YB-operator on $V$ of dimension $d$.

The (n-strand) braid group $B_n$ is defined as the group generated by $\sigma_1, \sigma_2, \ldots, \sigma_{n-1}$ satisfying:

$$ \sigma_i \sigma_j=\sigma_j\sigma_i \hspace{3mm} \text{for} \:\:|i-j| \geq 2 ;$$
$$ \sigma_i\sigma_{i+1}\sigma_i=\sigma_{i+1}\sigma_i\sigma_{i+1}\hspace{3mm} \text{for}\:\: i=1,2,\ldots, n-2.$$

Each gYB-operator gives rise to a representation of braid group $B_n \rightarrow \End(V^{\otimes k+m(n-2)})$ via $\sigma_i \mapsto R_i=I_m^{\otimes i-1}\otimes R \otimes I_m^{\otimes n-i-1}$. We denote this representation by $\rho^R_n$ and its image by $\im(\rho^R_n)$.

For any vector space one may consider a trace inner product $\langle f,g\rangle$ on $\End(V)$ defined by $\langle f,g\rangle=\tr(f^*\circ g)$ where $f^*$ is the hermitian conjugate of $f$. For any subset $A$ of $\End(V)$, we denote by $A^\bot$ the perpendicular subspace of A in $\End(V)$ with respect to this trace inner product. That is, for any $f \in A$ and $g \in A^\bot$, $\langle f,g\rangle=\tr(f^*\circ g)=0$.

\begin{defn}\label{def-EgYB}
Fix positive integers $k$ and $m$ such that $m<k$. An \textbf{enhanced generalized Yang-Baxter operator} (EgYB-operator) is a collection \{a gYB-operator $R:V^{\otimes k} \rightarrow V^{\otimes k}$, $\mu :V \rightarrow V $, invertible elements $\alpha, \beta$ of $\mathbb{C}$ \} which satisfies the following conditions for all $n$:

(i) The endomorphism $\mu^{\otimes k}  : V^{\otimes k} \rightarrow V^{\otimes k}$ commutes with $R$;

(ii) $\mu^{\otimes m(n-1)} \otimes (\Sp_{k,m}(R\circ \mu^{\otimes k})-\alpha\beta\mu^{\otimes k-m}) \in \im(\rho^R_n)^\bot$;\\ $\mu^{\otimes m(n-1)} \otimes (\Sp_{k,m}(R^{-1}\circ \mu^{\otimes k})-\alpha^{-1}\beta\mu^{\otimes k-m}) \in \im(\rho^R_n)^\bot$.

\end{defn}

\begin{rmk} The condition (ii) above is strictly weaker than the condition (ii) in the Definition \ref{def-EYB}. Indeed, the EYB-operator is the case that $k=2$, $m=1$, and $\mu^{\otimes m(n-1)} \otimes (\Sp_{k,m}(R^{\pm 1}\circ \mu^{\otimes k})-\alpha^{\pm 1}\beta\mu^{\otimes k-m})$ is the trivial endomorphism 0 which obviously belongs to $\im(\rho^R_n)^\bot$. We will see some nontrivial cases in the subsection \ref{ex-1}.
\end{rmk}

\section{Invariants of braids and links}\label{inv}

It is well known that any oriented link can be obtained by closing a braid, that is by connecting top endpoints with bottom endpoints by disjoint arcs (Alexander's theorem), and two braids produce isotopic links in this way if and only if these braids are related by a finite sequence of Markov moves $\xi \mapsto \eta^{-1}\xi\eta, \xi \mapsto \xi\sigma_n^{\pm1}$ where $\xi, \eta \in B_n$ (Markov's theorem). (see Chapter 2 in \cite{Bi})

\subsection{Invariants of braids}\label{inv-b}

We call the braid generators $\sigma_i$ positive crossings and their inverses negative crossings. For each braid $\xi \in B_n$, we denote the number of positive crossings by $w_+(\xi)$ and number of negative crossings by $w_-(\xi)$. Then $w(\xi)=w_+(\xi)-w_-(\xi)$ is a homomorphism from $B_n$ to the additive group of integers.

Each EgYB operator $S=(R,\mu,\alpha,\beta)$ determines a map $T_S: \bigsqcup_{n \geq 1}B_n\rightarrow \mathbb{C}$ as follows. For each braid $\xi \in B_n$ let

$$T_S(\xi)=\alpha^{-w(\xi)}\beta^{-n} \tr(\rho^R_n(\xi)\circ \mu^{\otimes k+m(n-2)}).$$

\noindent The following theorem provides key properties of $T_S$.

\begin{thm}\label{key} For any $\xi, \eta \in B_n$

$$T_S(\eta^{-1}\xi\eta)=T_S(\xi \sigma_n)= T_S(\xi \sigma_n^{-1})=T_S(\xi).$$

\end{thm}

\begin{proof}
 Let $N=k+m(n-2)$. From the definition of EgYB-operator, $\mu^{\otimes N}$ commutes with $\rho^R_n(\eta)$. Thus

$$\tr(\rho^R_n(\eta^{-1}\xi\eta)\circ \mu^{\otimes N} )=\tr(\rho^R_n(\eta^{-1})\circ \rho^R_n(\xi)\circ \mu^{\otimes N} \circ \rho^R_n(\eta) )=\tr(\rho^R_n(\xi)\circ \mu^{\otimes N}).$$

\noindent Since $w(\eta^{-1}\xi\eta)=w(\xi)$, we have $T_S(\eta^{-1}\xi\eta)=T_S(\xi)$.

Now we prove $T_S(\xi \sigma_n)=T_S(\xi)$. Note that for $\xi \in B_n$

$$\rho^R_{n+1}(\xi \sigma_n)=(\rho^R_n(\xi)\otimes I_m)\circ R_n, \:\:\: \text{and}$$

$$\mu^{\otimes N+m}=(I_{N+m-k} \otimes \mu^{\otimes k})\circ (\mu^{\otimes N+m-k}\otimes I_k)$$

\noindent

Thus,

\vspace{2mm}$\tr \left(\rho^R_{n+1}(\xi \sigma_n)\circ \mu^{\otimes N+m} \right)\\
=\tr \left( (\rho^R_n(\xi)\otimes I_m)\circ R_n \circ (I_{N+m-k} \otimes \mu^{\otimes k})\circ (\mu^{\otimes N+m-k}\otimes I_k) \right) \\
=\tr \left( \Sp_{N+m,m} \left[ (\rho^R_n(\xi)\otimes I_m)\circ R_n \circ (I_{N+m-k} \otimes \mu^{\otimes k})\circ (\mu^{\otimes N+m-k}\otimes I_k) \right] \right) \\
= \tr \left( \rho^R_n(\xi) \circ \left[ I_{N+m-k} \otimes \Sp_{k,m}(R\circ \mu^{\otimes k})  \right] \circ (\mu^{\otimes N+m-k}\otimes I_{k-m})  \right) $

\vspace{2mm}\noindent where the third equality comes from the Lemma \ref{sp}. By the definition of the EgYB-operator,

\vspace{2mm}$\tr \left(\rho^R_{n+1}(\xi \sigma_n)\circ \mu^{\otimes N+m} \right) - \alpha\beta \tr \left(  \rho^R_n(\xi) \circ \mu^{\otimes N} \right) \\= \tr \left( \rho^R_n(\xi) \circ \left[ I_{N+m-k} \otimes (\Sp_{k,m}(R\circ \mu^{\otimes k})- \alpha\beta\mu^{\otimes k-m}     )  \right]     \circ (\mu^{\otimes N+m-k}\otimes I_{k-m}) \right) \\
= \tr \left( \rho^R_n(\xi) \circ \left[ \mu^{\otimes N+m-k} \otimes (\Sp_{k,m}(R\circ \mu^{\otimes k})- \alpha\beta\mu^{\otimes k-m}     )  \right]      \right)=0$

\vspace{2mm}Clearly, $w(\xi \sigma_n)=w(\xi)+1$ and thus $T_S(\xi \sigma_n)=T_S(\xi)$. The other equality $T_S(\xi \sigma_n^{-1})=T_S(\xi)$ is obtained similarly.

\end{proof}

\subsection{Invariants of links}\label{inv-l}

Using the same notation $T_S$, we define a map from the set of isotopy classes of oriented links to $\mathbb{C}$ by $T_S(L)=T_S(\xi)$ where a link $L$ is isotopic to the closure of a braid $\xi$.  Markov's theorem and Theorem \ref{key} show that this map on oriented links is well defined and isotopy invariant.

%It is easy to see that the invariant $T_S$ is multiplicative, that is $T_S(L_1 \bigsqcup L_2)=T_S(L_1)\cdot T_S(L_2)$ for two disjoint links $L_1$ and $L_2$, essentially because $\tr(f\otimes g)=\tr(f)\cdot\tr(g)$.

We denote the trivial knot by $G_1$ and the trivial $n$-component link by $G_n$. For a $(d,k,m)$ EgYB-operator $S$, we have

\begin{equation}\label{gn} T_S(G_1)=\beta^{-1}\tr(\mu)^{k-m}\:\: \text{and}\:\:\: T_S(G_n)=\beta^{-n}\tr(\mu)^{k+m(n-2)}.
\end{equation}

  We say that an invariant of links, say $\mathcal{I}$, is multiplicative if $\mathcal{I}(L)=\mathcal{I}(L_1)\cdot \mathcal{I}(L_2)$ for disjoint union $L=L_1 \bigsqcup L_2$ of two links $L_1$ and $L_2$. From the formula (\ref{gn}), it is easy to see that the link invariant $T_S$ obtained from EgYB-operator is not necessarily multiplicative while the invariant from EYB-operator is so (see \cite{Tu}). This is essentially because $\rho^R_n(\sigma_i)$ and $\rho^R_n(\sigma_{i+2})$ possibly act nontrivially on the same tensor factor in the middle and as a result the endomorphism corresponding to a link $L=L_1 \bigsqcup L_2$ is not necessarily expressed as a tensor product of two operators in $\End(V^{\otimes k+m(n-2)})$. More specifically, this is the case when $m < \frac{1}{2}k$. However, link invariant $T_S$ obtained from any EgYB-operator $S$ is projectively multiplicative, and hence multiplicative after a normalization.

 \begin{thm}\label{mul}   If $S$ is a $(d,k,m)$ EgYB-operator for any $m$ and $k$, then the invariant $T_S$ is projectively multiplicative:

  $$T_S(L_1\bigsqcup L_2)=\tr(\mu)^{2m-k} T_S(L_1)\cdot T_S(L_2)$$

 \end{thm}

 \noindent Proof is given in the Appendix. Note that in the case of EYB-operators, $m=1$ and $k=2$ and thus the corresponding invariant is multiplicative without the factor in equation (\ref{mul}).

 \begin{rmk} It is easy to see that

 $$ \tilde{T}_{S}(L)=\tr(\mu)^{2m-k}T_S(L) $$

 is multiplicative. Indeed, $\tilde{T}_{S}(L_1\bigsqcup L_2)=\tilde{T}_S(L_1)\cdot \tilde{T}_S(L_2)$.
 \end{rmk}

For a link diagram $L_0$, we may consider two link diagrams $L_+$ and $L_-$ which are obtained by a local deformation introducing a positive and a negative crossing, respectively (see Figure\ref{Conway}).

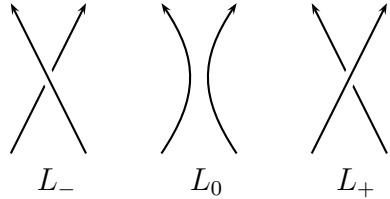
\begin{figure}[h]\psset{unit=5mm}

$\begin{pspicture}(0,1)(2,5)
\psline{<-}(2,5)(0,1)
\psline[linecolor=white,linewidth=5pt](0.5,4)(1.5,2)
\psline{<-}(0,5)(2,1)
\end{pspicture}
\hspace{10mm}
\begin{pspicture}(0,1)(2,5)
\psbezier{->}(0,1)(1,2.5)(1,3.5)(0,5)
\psbezier{->}(2,1)(1,2.5)(1,3.5)(2,5)
\end{pspicture}
\hspace{10mm}
\begin{pspicture}(0,1)(2,5)
\psline{<-}(0,5)(2,1)
\psline[linecolor=white,linewidth=5pt](1.5,4)(0.5,2)
\psline{<-}(2,5)(0,1)
\end{pspicture}$

$\hspace{2mm}L_{-}\hspace{15mm}L_0 \hspace{15mm}L_{+}$

\caption{\small Link diagrams $L_+$ and $L_-$ are obtained from $L_0$ by introducing a positive and a negative crossing, respectively, in some disk. Outsides the disk are identical.  }
\label{Conway}
\end{figure}

 A Conway-type relation between the invariants of links $L_-$,$L_0$, and $L_+$ is particularly interesting and the following theorem is reproved in \cite{Tu}.

\begin{thm}\label{unique}
There exists a unique mapping $P$ from the set of isotopy types of oriented links into the ring $\mathbb{Z}[x,x^{-1},y,y^{-1}]$ such that $P(G_1)=1$ and for any triple $(L_-,L_0,L_+)$

$$xP(L_+)+x^{-1}P(L_-)=yP(L_0).$$
\end{thm}

\section{Examples}\label{ex}

At present not many gYB-operators are known. A(2,3,2) gYB-operator is given in \cite{RZWG} and a (2,3,1) gYB-operator is in \cite{GHR}. Three families of variations of the latter are obtained in \cite{Ch} and we will call those as (2,3,1) gYB-operator of type I, II, and III in the following subsection \ref{ex-1}. All these gYB-operators can be enhanced and give rise to invariants of oriented links.

\subsection{(2,3,1)EgYB-operators}\label{ex-1}

A vector space $V$ is 2-dimensional. On the lexicographically ordered basis of $V^{\otimes 3}$, three families of (2,3,1) gYB-operators are represented as $8 \times 8$ unitary matrices for $0 \leq \theta \leq \pi$ as follows \cite{Ch}:

\vspace{2mm}
Type I: $R(\theta)=\frac{1}{\sqrt{2}}\begin{pmatrix}1&0&1&0\\ 0&i&0&e^{i\theta}\\-i&0&i&0\\0&-ie^{-i\theta}&0&1      \end{pmatrix} \oplus \frac{1}{\sqrt{2}}\begin{pmatrix}i&0&e^{i\theta}&0\\ 0&1&0&-e^{2i\theta}\\-ie^{-i\theta}&0&1&0\\0&ie^{-2i\theta}&0&i      \end{pmatrix}$

\vspace{2mm}
Type II: $R(\theta)=\frac{1}{\sqrt{2}}\begin{pmatrix}1&0&1&0\\ 0&i&0&e^{i\theta}\\-1&0&1&0\\0&e^{-i\theta}&0&i      \end{pmatrix} \oplus \frac{1}{\sqrt{2}}\begin{pmatrix}i&0&e^{i\theta}&0\\ 0&1&0&-e^{2i\theta}\\e^{-i\theta}&0&i&0\\0&e^{-2i\theta}&0&1      \end{pmatrix}$

\vspace{2mm}
Type III: $R(\theta)=\frac{1}{\sqrt{2}}\begin{pmatrix}1&0&1&0\\ 0&1&0&e^{i\theta}\\-1&0&1&0\\0&-e^{-i\theta}&0&1      \end{pmatrix} \oplus \frac{1}{\sqrt{2}}\begin{pmatrix}1&0&-e^{i\theta}&0\\ 0&1&0&-e^{2i\theta}\\e^{-i\theta}&0&1&0\\0&e^{-2i\theta}&0&1      \end{pmatrix}$

\begin{rmk}\label{rmk}
\begin{enumerate}
\item Note that for all of three families $R^{j_1,j_2,j_3}_{i_1,i_2,i_3}=0$ if $i_1\neq j_1$ or $i_3\neq j_3$ where $i_1,i_2,i_3,j_1,j_2,j_3 \in \{1,2\}$. That is, $R$ acts on the first and third tensor factors diagonally, and so does $R^{-1}$. Which means in particular that $\left(\rho^R_n(\xi)\right)^{j_1,\ldots,j_{N-1},j_{N}}_{i_1,\ldots,i_{N-1},i_{N}}=0$ if $i_N\neq j_N$ for any $\xi \in B_n$ where $N=k+m(n-2)=n+1$.
\item For $f,g \in \End(V^{\otimes N})$, it is well known that

$$\tr(f\circ g)= \sum_{1 \leq i_1,j_1,\ldots, i_N,j_N \leq \dim(V)}f^{i_1,\ldots,i_{N-1},i_{N}}_{j_1,\ldots,j_{N-1},j_{N}}g^{j_1,\ldots,j_{N-1},j_{N}}_{i_1,\ldots,i_{N-1},i_{N}}.$$

\item Suppose that $g \in \End(V^{\otimes N})$ acts on the last tensor factor off-diagonally. Explicitly, that is,
$g^{j_1,\ldots,j_{N-1},j_{N}}_{i_1,\ldots,i_{N-1},i_{N}}=0$ if $i_N = j_N$. Then by combining two facts above, one obtains the following:

$$\tr\left(\rho^R_n(\xi) \circ g \right)=0 \:\: \text{for any}\:\: \xi \in B_n $$

\noindent and which means that $g \in \im(\rho^R_n)^\bot.$

\end{enumerate}
\end{rmk}

\vspace{3mm}\noindent \emph{\underline{(2,3,1) EgYB-operator of type I}}

\begin{thm}\label{thm-I}
Let $R$ be a (2,3,1) gYB-operator of type I. Let $\mu=\Id_{V}$, $\alpha=e^{i\pi/4}$, and $\beta=1$. Then $S=(R,\mu,\alpha,\beta)$ is an EgYB-operator such that

$$T_S(L_{+})+T_S(L_{-})=T_S(L_0)$$

and $T_S(G_n)=2^{n+1}$.
\end{thm}

\begin{proof}
The commutativity condition (i) in Definition \ref{def-EgYB} is trivial because $\mu$ is the identity.
For the condition (ii), we need to show

$\tr \left[ \rho^R_n(\xi) \circ \left( \Id_V^{\otimes m(n-1)}\otimes (\Sp_{3,1}(R^{\pm 1})-\alpha^{\pm 1}\Id_V^{\otimes 2}) \right) \right]=0$

\vspace{3mm}\noindent for any $\xi \in B_n$. From the Remark \ref{rmk}, it suffices to show that $\Sp_{3,1}(R^{\pm 1})-\alpha^{\pm 1}\Id_V^{\otimes 2}$ acts on the last tensor factor off-diagonally. A direct computation gives us the following:

$\Sp_{3,1}(R)-\alpha\beta \Id_V^{\otimes 2}=\frac{1}{\sqrt{2}}\begin{pmatrix}0&1+e^{i\theta}\\-i-ie^{-i\theta}&0\end{pmatrix}\oplus \frac{1}{\sqrt{2}}\begin{pmatrix}0&e^{i\theta}-e^{2i\theta}\\-ie^{-i\theta}+ie^{-2i\theta}&0\end{pmatrix}$

\noindent which shows off-diagonal action on the second tensor factor. $R^{-1}$ case is easily derived from the above by complex conjugating. Hence, $S=(R,\mu,\alpha,\beta)$ is an EgYB-operator.

Now let us prove the Conway-type relation. For any triple $(L_-,L_0,L_+)$ we may choose a braid $\xi \in B_n$ such that the link $L_0$ is isotopic to the closure of $\xi$, and links $L_{\pm}$ are isotopic to the closures of $\sigma^{\pm}_{1} \xi$, respectively. The following computation comes directly from the definition of $T_S$ and the fact $e^{-i\pi/4}R+e^{i\pi/4}R^{-1}=\Id_{V^{\otimes 3}}$.

\vspace{3mm}$T_S(L_+)+T_S(L_-)=T_S(\sigma_{1} \xi)+T_S(\sigma^{-1}_{1} \xi)\\= \alpha^{-w-1}\tr \left((R\otimes \Id_V^{\otimes n-2})\circ \rho^R_n(\xi) \right)+\alpha^{-w+1}\tr \left((R^{-1}\otimes \Id_V^{\otimes n-2})\circ \rho^R_n(\xi) \right)\\
= \alpha^{-w} \tr \left( (\alpha^{-1}R+\alpha R^{-1})\otimes \Id_V^{\otimes n-2}) \circ \rho^R_n(\xi) \right)\\
=\alpha^{-w} \tr \left( (\Id_V^{\otimes 3}\otimes \Id_V^{\otimes n-2}) \circ \rho^R_n(\xi) \right)= T_S(\xi)=T_S(L_0)$

\vspace{3mm}\noindent where $w=w(\xi)$.

The last statement is obtained from the formula (\ref{gn}) and $\tr(\mu)=\tr(\Id_V)=2$.

\end{proof}

It is well known that any link diagram can be transformed into a diagram of trivial link via finitely many crossing changes, that is, via replacing some positive crossings with negative crossings and vice versa. If we apply the Conway-type relation in Theorem \ref{thm-I} for each step of such process, we obtain that for any link $L$ the invariant $T_S(L)$ is a finite sum of $T_S(G_d)$'s with integer coefficients. More specifically, the values $T_S(L)$ are multiples of 4 because the values of $T_S(G_d)$ are so. Let us normalize invariant $T_S$ as follows:

$$P_S(L)=\frac{1}{4}T_S(L)$$

\noindent Then this invariant $P_S$ has the same Conway-type relation as given in the theorem above and $P_S(G_1)=1$. This is the case of invariant $P$ in Theorem \ref{unique} with $x=y=1$, and by uniqueness it is a special case of the invariant $P$ appeared in \cite{Tu}.

\vspace{3mm}\noindent \emph{\underline{(2,3,1) EgYB-operator of type II}}

\begin{thm}
Let $R$ be a (2,3,1) gYB-operator of type II. Let $\mu=\Id_{V}$, $\alpha=e^{i\pi/4}$, and $\beta=1$. Then $S=(R,\mu,\alpha,\beta)$ is an EgYB-operator such that $T_S(G_n)=2^{n+1}$.
\end{thm}

\begin{proof}

The first statement is proved in the same way as Theorem \ref{thm-I} with the following:

$\Sp_{3,1}(R)-\alpha\Id_V^{\otimes 2}=\frac{1}{\sqrt{2}}\begin{pmatrix}0&1+e^{i\theta}\\-1+e^{-i\theta}&0\end{pmatrix}\oplus \frac{1}{\sqrt{2}}\begin{pmatrix}0&e^{i\theta}-e^{2i\theta}\\e^{-i\theta}+e^{-2i\theta}&0\end{pmatrix}$

\noindent The formula for trivial links is directly from the formula (\ref{gn}) and $\tr(\mu)=\tr(\Id_V)=2$.

\end{proof}

In this case, there is no Conway-type relation as the corresponding minimal polynomial is of degree 3. Instead we have a relation

\begin{equation}\label{+2} T_S(L_{+2})-T_S(L_{+})+T_S(L_{0})-T_S(L_{-})=0 \end{equation}

\noindent from $\alpha^{-2}R^2-\alpha^{-1} R+ \Id-\alpha R^{-1}=0$ where $L_{+2}$ denotes the link containing two positive crossings inside the disk we considered in Figure \ref{Conway}. A direct consequence of equation (\ref{+2}) is $T_S(H)=0$ where $H$ denotes the Hopf link. More generally, we consider any disjoint union link $L_1 \bigsqcup L_2$ as $L_0$, and construct $L_{+2}$ by braiding two parallel strands twice where one of the two strands should be taken from $L_1$ and the other from $L_2$. Then we still obtain $T_S(L_{+2})=0$. It does not come directly from the equation (\ref{+2}) though. One way to see this is the following: Any such a link $L_{+2}$ can be represented as the closure of a braid which has $\sigma^2_i$ for some $i$ and no more $\sigma^{\pm 1}_i$ before and after. The corresponding image under the representation $\rho^R$ will have $R^2_i$ and no more $R^{\pm 1}_i$ before and after. Now observe that $R^2$ acts off-diagonally on the second tensor factor while $R^{\pm 1}$ acts diagonally on the first and third tensor factors. As a result, trace of the corresponding image is zero.

%One may rather expect a Kauffman-type relation which is on invariants of non-oriented links. However it does not seem to exist either. We may obtain an invariant of non-oriented links by introducing extra normalizing factor based on self-crossing writhe or non-self-crossing writhe (see, for example, \cite{Tu}). Another way is by considering sum over all possible orientations (see, for example, \cite{Pr}). Let us denote an invariant of non-oriented links by bracket $[L]$. Kauffman-type relation is of the form: $[L_+]+a[L_-]=b\left( [L_0]+a[L_{\infty}] \right)$. Using a few simple links (for example, trivial knot $G_1$, Hopf link $H$ and trefoil knot) it is easy to see that there is no solution for $a$ and $b$.

$T_S(L)=4$ for a few simplest knots. And for a few simple links with two components, the values are either 0 or 8 depending on the linking numbers. It seems that the invariant $T_S$ depends only on the number of components and the linking numbers.

Eric C. Rowell pointed out to me that the braid group representation $\rho^{R}$ from type II can be seen to factor over the BMW-algebra with parameters $r=q=e^{\pi i/4}$ (see \cite{We}). This suggests a connection to a specialization of the Kauffman polynomial $F$.

\vspace{3mm}\noindent \emph{\underline{(2,3,1) EgYB-operator of type III}}

\begin{thm}
Let $R$ be a (2,3,1) gYB-operator of type III. Let $\mu=\Id_{V}$, $\alpha=1$, and $\beta=\sqrt{2}$. Then $S=(R,\mu,\alpha,\beta)$ is an EgYB-operator such that

$$T_S(L_{+})+T_S(L_{-})=\sqrt{2}\:\:T_S(L_0)$$

and $T_S(G_n)=(\sqrt{2})^{-n}2^{n+1}$.
\end{thm}

\begin{proof}

The first statement is proved in the same way as Theorem \ref{thm-I} with the following:

$\Sp_{3,1}(R)-\alpha\beta\Id_V^{\otimes 2}=\frac{1}{\sqrt{2}}\begin{pmatrix}0&1+e^{i\theta}\\-1-e^{-i\theta}&0\end{pmatrix}\oplus \frac{1}{\sqrt{2}}\begin{pmatrix}0&-e^{i\theta}-e^{2i\theta}\\e^{-i\theta}+e^{-2i\theta}&0\end{pmatrix}$

\noindent The Conway-type relation can be easily shown in the same way as Theorem \ref{thm-I} from a relation
$R+R^{-1}=\sqrt{2}\Id_{V^{\otimes 3}}.$ The last statement is a direct consequence of the formula (\ref{gn}) and $\tr(\mu)=\tr(\Id_V)=2$.

\end{proof}

By defining $P_S(L)=\frac{1}{2\sqrt{2}}T_S(L)$ with $x=1, y=\sqrt{2}$, we see that the invariant $P_S$ satisfies the conditions in Theorem \ref{unique} and thus this is a special case of the invariant $P$ appeared in \cite{Tu} as well.

\subsection{(2,3,2)EgYB-operator}\label{ex-2}

A vector space $V$ is 2-dimensional. In \cite{RZWG} it is shown that $R=\frac{1}{\sqrt{2}}\begin{pmatrix} I&J\\-J&I \end{pmatrix}$ is a unitary (2,3,2) gYB-operator where $I$ is the $4\times4$ identity matrix and $J=\begin{pmatrix} 0&0&0&1\\0&0&1&0\\0&1&0&0\\1&0&0&0 \end{pmatrix}$.

\begin{thm}
Let $R$ be the (2,3,2) gYB-operator as above. Let $\mu=\Id_{V}$, $\alpha=1$, and $\beta=2\sqrt{2}$. Then $S=(R,\mu,\alpha,\beta)$ is an EgYB-operator such that

$$T_S(L_{+})+T_S(L_{-})= \sqrt{2}\:\:T_S(L_0)$$

and $T_S(G_n)=(\sqrt{2})^{-n}2^{n-1}$.

\end{thm}

\begin{proof}

The commutativity condition (i) in Definition \ref{def-EgYB} is trivial because $\mu$ is the identity.
For the condition (ii) a direct computation shows that $\Sp_{k,m}(R^{\pm 1}\circ \mu^{\otimes k})-\alpha^{\pm 1}\beta\mu^{\otimes k-m}=0$. Indeed $\Sp_{3,2}(R^{\pm 1})= 2\sqrt{2}\Id_{V}$ and hence $S=(R,\mu,\alpha,\beta)$ is an EgYB-operator.

For the Conway-type relation, observe that $R+R^{-1}=\sqrt{2}\Id_{V^{\otimes 3}}$. The last statement is from the formula (\ref{gn}) and $\tr(\mu)=2$.

\end{proof}

 A normalized invariant $P_S=\sqrt{2} T_S$ satisfies the conditions in Theorem \ref{unique} with $x=1, y=\sqrt{2}$. So this is a special case of the invariant $P$ in \cite{Tu}. Indeed, the link invariants from $(2,3,1)$ EgYB-operator of type III and $(2,3,2)$ EgYB-operator are the same up to a factor 4. If we denote the former by $T_{S(III)}$ and the latter by $T_{S(2,3,2)}$ then $\frac{1}{4}T_{S(III)}(L)=T_{S(2,3,2)}(L)$ for any link $L$.

\section{Appendix: Proof of Theorem \ref{mul}}

In this appendix, we prove that the link invariant $T_S(L)$ is projectively multiplicative for any EgYB-operator $S$, and hence multiplicative after a normalization.

Let a link $L=L_1\bigsqcup L_2$ and $L, L_1, L_2$ be isotopic to the closure of $\xi=\xi_1 * \xi_2 \in B_{n}$, $\xi_1 \in B_{n_1}$, $\xi_2 \in B_{n_2}$, respectively, where $*$ denotes juxtaposition and $n=n_1+n_2$. Let $S=(R,\mu,\alpha,\beta)$ be a $(d,k,m)$ EgYB-operator. Choose a nonnegative integer $l$ such that $ml+2m-k \geq 0$ (if $2m-k\geq 0$, then one may choose $l=0$). We consider a link $L'=L\bigsqcup G_l$ which is obtained by inserting disjoint $l$ trivial circles into $L$. Then $L'$ is isotopic to the closures of both $\xi^{(1)}=\xi_1 * \mid_l *  \xi_2 \in B_{n+l}$ and $\xi^{(2)}=\xi_1  *  \xi_2 * \mid_l \in B_{n+l}$ where $\mid_l \in B_l$ is the identity. We obtain $T_S(L')$ from these two braids separately and then compare them. At first, we obtain the following from $\xi^{(1)}$

\vspace{3mm}
$\tr \left( \rho^R_{n+l}(\xi^{(1)}) \circ \mu^{\otimes k+m(n+l-2)} \right)
= \tr \left( \rho^R_{n+l}(\xi_1 * \mid_l *  \xi_2) \circ \mu^{\otimes k+m(n+l-2)} \right)\\
= \tr \left( \left[\rho^R_{n_1}(\xi_{1})\otimes \Id_V^{\otimes m(n_2+l)} \right]\circ  \left[ \Id_V^{\otimes m(n_1+l)}\otimes \rho^R_{n_2}(\xi_{2})\right] \circ  \mu^{\otimes k+m(n+l-2)} \right) \\
= \tr\left( \left[ \rho^R_{n_1}(\xi_{1}) \circ \mu^{\otimes k+m(n_1-2)} \right] \otimes \left[ \mu^{\otimes m(l+2)-k}  \right] \otimes \left[ \rho^R_{n_2}(\xi_{2}) \circ \mu^{\otimes k+m(n_2-2)} \right]    \right) \\
=\tr(\mu)^{m(l+2)-k} \cdot \tr\left(\rho^R_{n_1}(\xi_{1}) \circ \mu^{\otimes k+m(n_1-2)}\right)  \cdot \tr\left(\rho^R_{n_2}(\xi_{2}) \circ \mu^{\otimes k+m(n_2-2)}\right).$

\vspace{3mm}\noindent On the other hand, from $\xi^{(2)}$

$\tr \left( \rho^R_{n+l}(\xi^{(2)}) \circ \mu^{\otimes k+m(n+l-2)} \right)
= \tr \left( \rho^R_{n+l}(\xi_1  *  \xi_2 * \mid_l) \circ \mu^{\otimes k+m(n+l-2)} \right)\\
= \tr \left( \left[\rho^R_{n}(\xi)\otimes \Id_V^{\otimes ml} \right]\circ \mu^{\otimes k+m(n+l-2)}\right) \\
= \tr(\mu)^{ml} \cdot \tr \left( \rho^R_{n}(\xi) \circ \mu^{\otimes k+m(n-2)}\right )$

\vspace{3mm} \noindent Thus

$\tr \left( \rho^R_{n}(\xi) \circ \mu^{\otimes k+m(n-2)}\right )=\tr(\mu)^{2m-k} \cdot \tr\left(\rho^R_{n_1}(\xi_{1}) \circ \mu^{\otimes k+m(n_1-2)}\right)  \cdot \tr\left(\rho^R_{n_2}(\xi_{2}) \circ \mu^{\otimes k+m(n_2-2)}\right) $.

\noindent Because $w(\xi)=w(\xi_1)+w(\xi_2)$ and $n=n_1+n_2$, it implies

$T_S(L)=T_S(\xi)=\tr(\mu)^{2m-k} \cdot T_S(\xi_1) \cdot T_S(\xi_2) = \tr(\mu)^{2m-k} \cdot T_S(L_1) \cdot T_S(L_2).$

%what is special after doubling?
%why my work is interesting?
%What can I do further using the data done
%further research: physics, computer

%Check publications such as Mugher using MathSciNet, some of them are published.


\begin{thebibliography}{AA}

\bibitem[Bi] {Bi} J. S. Birman, \emph{Braids, links, and mapping class groups}, Princeton University Press 1974.

\bibitem[Ch] {Ch} R. S. Chen, \emph{Generalized Yang-Baxter equations and braiding quantum gates}, to appear in J. Knot Theory Ramifications, arXiv:1108.5215.

\bibitem[GHR]{GHR} C. Galindo, S.-M. Hong, E. Rowell, \emph{Generalized and quasi-localizations of braid group representations},  to appear in Int. Math. Res. Not., arXiv:1105.5048.



\bibitem[RZWG]{RZWG} E.\ Rowell, Y.\ Zhang, Y.-S.\ Wu and M.-L.\ Ge, \emph{Extraspecial two-goups, generalized Yang-Baxter equations and braiding quantum gates}, Quantum Inf. Comput. 10(2010) no. 7-8, 0685-0702, arXiv:quant-ph/0706.1761.

%\bibitem[Pr]{Pr} J.\ H.\ Przytychi, \emph{A note on the Kickorish-Millett-Turaev formula for the Kauffman polynomial}, Proc. Amer. Math. Soc. 121 no. 2, 645-647.

\bibitem[Tu]{Tu} V. G. Turaev, \emph{The Yang-Baxter equation and invariants of links}, Invent. Math. 92 (1988), 527-553.

\bibitem[We]{We} H. Wenzl, \emph{Quantum groups and subfactors of type B, C, and D}, Comm. Math. Phys. 133 (1990), 383-432.


\end{thebibliography}
\end{document}